\documentclass[11pt]{article}
\usepackage[a4paper,margin=25mm]{geometry}
\usepackage{amsmath,amssymb,amsthm,mathtools}
\usepackage[pdfusetitle,colorlinks=true,urlcolor=blue,citecolor=blue,linkcolor=blue]{hyperref}
\newtheorem{theorem}{Theorem}
\newtheorem{lemma}{Lemma}

\newcommand{\Q}{\mathbb Q}
\newcommand{\C}{\mathbb C}
\newcommand{\Z}{\mathbb Z}

\title{A simpler proof that there are no Landau--Siegel zeros}
\author{Bryce Kerr}

\date{}
\begin{document}
\maketitle

\begin{abstract}

This paper simplifies the recent OpenAI proof that there are no Landau--Siegel zeros.
\end{abstract}
\section{Introduction}

OpenAI~\cite{OpenAI} recently established that Siegel zeros do not exist. Our aim is to give a conceptually simpler proof of this result.

\begin{theorem}\label{thm:nosiegel}
There is an absolute constant $c>0$ such that, for every primitive nonprincipal real Dirichlet character $\chi(n)=(q/n)$, where $q$ is a fundamental discriminant, every real zero $\beta\in(0,1)$ of $L(s,\chi)$,
\[
 (1-\beta)\log|q|\ge c.
\]
\end{theorem}

The argument has much in common with OpenAI~\cite{OpenAI} who consider polynomials that vanish modulo many primes when evaluated at linear forms. Our main difference is to replace the interpolation and determinant analysis with the pigeonhole principle to find a polynomial $f$ of bounded degree which vanishes at our linear forms. A result of Masser~\cite{Masser} then contradicts our degree bound for $f$. 

\paragraph{Acknowledgement} ChatGPT assisted with the mathematical development and preparation of this paper. The author thanks  Bill Banks for sharing a draft of his exposition of the OpenAI proof~\cite{Banks} which led to the development of this manuscript.

\paragraph{Analytic input}
We will ultimately derive a contradiction from the following result, which is a straightforward consequence of~\cite[Equation (2), p.~89]{Davenport}.
\begin{lemma}\label{analytic}

Let $\beta\in(0,1)$ be a real zero of $L(s,\chi)$. For every $X\ge3$,
\begin{equation}\label{bias}
 \sum_{\substack{p\le X\\\chi(p)=1}}\frac{\log p}{p}
 \ll \log|q|+(1-\beta)(\log X)^2.
\end{equation}
\end{lemma}

\paragraph{Number fields}
As in~\cite{OpenAI}, we work over the field $K$ and the ring $\mathcal O$ given by
\[
 K=\Q(\sqrt q,\sqrt2), \quad \mathcal O=\Z[\sqrt q,\sqrt2].
\]

 The field $K$ has degree four and $\mathcal O$ has $\Z$-basis $1,\sqrt q,\sqrt2,\sqrt{2q}$.  For $\alpha\in K$, let $\alpha^{(1)},\ldots,\alpha^{(4)}$
denote its images under the four embeddings $K\hookrightarrow\mathbb C$,
and define
\[
\|\alpha\|=\max_{1\le t\le4}|\alpha^{(t)}|.
\]

 Let $\sigma,\tau$ be the automorphisms determined by
\[
 \sigma(\sqrt q)=-\sqrt q,\quad \sigma(\sqrt2)=\sqrt2,
 \qquad \tau(\sqrt q)=\sqrt q,\quad \tau(\sqrt2)=-\sqrt2.
\]

For an integer $N\ge4$ and a vector $n\in\{0,\ldots,N-1\}^4$, write
\begin{align}
\label{eq:thetadef}
 \theta_n=n_1+n_2\sqrt q+n_3\sqrt2+n_4\sqrt{2q}.
\end{align}
For $p\nmid2q$ the action of Frobenius modulo $p$, see~\cite[\S8, Example 8.19]{Milne}, gives
\[
 (\sqrt q)^p\equiv\chi(p)\sqrt q,\qquad
 (\sqrt2)^p\equiv(2/p)\sqrt2\pmod{p\mathcal O}.
\]

Hence, if $\chi(p)=-1$ then
\begin{align}
 \theta_n^p&\equiv\sigma\theta_n\pmod{p\mathcal O}
 &&\text{if }(2/p)=1,\label{eq:2+1}\\
 \theta_n^p&\equiv\sigma\tau\theta_n\pmod{p\mathcal O}
 &&\text{if }(2/p)=-1.\label{eq:2-1}
\end{align}

\paragraph{Polynomials}
For an integer $D\ge0$, let $\mathcal L_D$ be the $\mathcal O$-lattice generated by the monomials 
\begin{align}
\label{eq:basis1}
x_1^ix_2^jx_3^k, \quad i+8(j+k)\le D.
\end{align}
 Thus its elements are
\begin{align}
\label{eq:f123}
f(x_1,x_2,x_3)=\sum_{i+8(j+k)\le D}c_{i,j,k}x_1^ix_2^jx_3^k,
\qquad c_{i,j,k}\in\mathcal O.
\end{align}
The following lemma provides an alternative basis for $\mathcal L_D$.

\begin{lemma}\label{basis}
For any prime $p>8$ and $\ell\in \{2,3\}$ the set of polynomials
\begin{align}
\label{eq:basis2}
 G_{i,j,k}(x_1,x_2,x_3)=x_1^v(x_1^p-x_{\ell})^e x_2^jx_3^k, \quad i=pe+v, \quad 0\le v<p, \quad i+8(j+k)\le D
\end{align}
forms a basis of $\mathcal L_D$.

\end{lemma}
\begin{proof}
This follows from the fact that the change of basis matrix from~\eqref{eq:basis1} to~\eqref{eq:basis2} is a
 triangular matrix with diagonal entries one.
\end{proof}

For $B\ge1$ and $f$ as in~\eqref{eq:f123} define
\[
\|f\|_B=
\max_{i+8(j+k)\le D}\|c_{ijk}\|B^{i+j+k}.
\]

We will apply the pigeonhole principle to sets of the form 
$\{f\in \mathcal L_D \ : \ \|f\|_B\le H\}.$
Our next result will be used to show that these sets are large enough for our pigeonhole application.
\begin{lemma}\label{coefficient-count}
For $H>0$ and $B\ge1$,
\[
\#\{f\in\mathcal L_D:\|f\|_B\le H\}
>
\prod_{i+8(j+k)\le D}
\left(\frac{H}{8\sqrt{|q|}\,B^{i+j+k}}\right)^4.
\]
\end{lemma}
\begin{proof}
In~\eqref{eq:f123} we take
\[
c_{i,j,k}=u_1+u_2\sqrt q+u_3\sqrt2+u_4\sqrt{2q},
\qquad
0\le u_t\le
\left\lfloor\frac{H}{8\sqrt{|q|}\,B^{i+j+k}}\right\rfloor.
\]
These choices satisfy $\|c_{ijk}\|B^{i+j+k}\le H$. Since there are
more than
$(H/(8\sqrt{|q|}\,B^{i+j+k}))^4$ choices for each coefficient the result follows.
\end{proof}

\paragraph{Vanishing modulo primes}

In order to contradict Lemma~\ref{analytic}, we will construct a polynomial from $\mathcal L_D$ which vanishes modulo many small primes $p$. Degree and height bounds will then force it to vanish over $\mathcal{O}$. As in~\cite{OpenAI}, a sufficient degree bound will crucially rely on evaluating $f(x_1,x_2,x_3)\in \mathcal L_D$ at points  
\begin{align}
\label{eq:linearformpoints}
(x_1,x_2,x_3)=(\theta_n,\sigma\theta_n,\sigma\tau\theta_n) \quad n\in\{0,\ldots,N-1\}^4.
\end{align}

Our next  lemma will be combined with the pigeonhole principle to find a nonzero \(f\) vanishing at all points~\eqref{eq:linearformpoints}. 

\begin{lemma}\label{evaluation-residues}
Let $\mathcal S$ be a finite set of primes $p>8$ satisfying
 $p\nmid2q$ and $\chi(p)=-1$. For each $p\in \mathcal S$, let $k_p$ be an integer satisfying $k_p\ge\max\{1,\lfloor D/p\rfloor\}$ and put
\[
 r=\operatorname{rank}_{\mathcal O}\mathcal L_D,\qquad
 E_p=\sum_{i+8(j+k)\le D}\left\lfloor\frac{i}{p}\right\rfloor,
 \qquad
 P=\prod_{p\in\mathcal S}p^{k_p}.
\] We have 
\[
\#\left\{
 \bigl(f(\theta_n,\sigma\theta_n,\sigma\tau\theta_n)
       \bmod P\mathcal O\bigr)_{n\in\{0,\ldots,N-1\}^4}
 : f\in\mathcal L_D
\right\}
\le P^{4r}\prod_{p\in\mathcal S}p^{-4E_p}.
\]
\end{lemma}
\begin{proof}
Fix $p\in\mathcal S$ and take $\ell=2$ or $3$ according as $(2/p)=1$ or $-1$. By Lemma~\ref{basis} each $f\in\mathcal L_D$ can be expressed in the form
\[
f(x_1,x_2,x_3)=\sum_{i+8(j+k)\le D}c_{i,j,k}G_{i,j,k}(x_1,x_2,x_3),
\qquad c_{i,j,k}\in\mathcal O.
\]
By~\eqref{eq:2+1} and \eqref{eq:2-1}, for each $n\in\{0,\ldots,N-1\}^4$ 
$$G_{ijk}(\theta_n,\sigma \theta_n, \sigma\tau \theta_n)\equiv 0 \pmod{p^{\lfloor i/p\rfloor}\mathcal O}.$$ Consequently, the values of
\[
c_{i,j,k}\bmod p^{k_p-\lfloor i/p\rfloor}\mathcal O
\]
determine the  vector 
$$\bigl(f(\theta_n,\sigma\theta_n,\sigma\tau\theta_n)
       \bmod p^{k_p}\mathcal O\bigr)_{n\in\{0,\ldots,N-1\}^4}.$$
Hence
\begin{align*}
\#\left\{
\bigl(f(\theta_n,\sigma\theta_n,\sigma\tau\theta_n)
\bmod p^{k_p}\mathcal O\bigr)_{n\in\{0,\ldots,N-1\}^4}
:f\in\mathcal L_D
\right\}
&\le
\prod_{i+8(j+k)\le D}
\#\left(\mathcal O/p^{k_p-\lfloor i/p\rfloor}\mathcal O\right)\\
&=
\prod_{i+8(j+k)\le D}p^{4(k_p-\lfloor i/p\rfloor)}
=p^{4k_pr-4E_p}
\end{align*}
 and the Chinese remainder theorem gives
\[
\#\left\{
\bigl(f(\theta_n,\sigma\theta_n,\sigma\tau\theta_n)
\bmod P\mathcal O\bigr)_{n\in\{0,\ldots,N-1\}^4}
:f\in\mathcal L_D
\right\}
\le
\prod_{p\in\mathcal S}p^{4k_pr-4E_p}
=P^{4r}\prod_{p\in\mathcal S}p^{-4E_p}.
\]
\end{proof}
\begin{lemma}\label{pigeonholing}
Fix a positive integer $\gamma$ and put
\begin{gather*}
N=|q|^\gamma,\qquad B=8N\sqrt{|q|},\qquad
D=\left\lfloor\frac12\left(\frac{N-1}{3}\right)^{4/3}\right\rfloor,\\
Q=\frac{D}{(\log N)^2},\qquad
\mathcal S=\{p:8<p\le Q,\ p\nmid2q,\ \chi(p)=-1\}.
\end{gather*}
For sufficiently large $N$, if
\begin{align}
\label{eq:PS}
\sum_{p\in\mathcal S}\frac{\log p}{p}
>\frac{13}{10}\log B,
\end{align}
then there exists a nonzero polynomial $F\in\mathcal L_D$ such that
\begin{align}
\label{eq:Fconds}
F(\theta_n,\sigma\theta_n,\sigma\tau\theta_n)=0
\qquad\text{for every }n\in\{0,\ldots,N-1\}^4.
\end{align}
\end{lemma}
\begin{proof}
Take $k_p=\max\{1,\lfloor D/p\rfloor\}$, and let $P,r,E_p$ be as in Lemma~\ref{evaluation-residues}. Put
\[
W=\sum_{i+8(j+k)\le D}(i+j+k),\qquad H=\frac{P}{4r}.
\]
By Lemma~\ref{coefficient-count}, the number of polynomials $f\in\mathcal L_D$ with $\|f\|_B\le H$ exceeds
\[
\left(\frac{P}{32r\sqrt{|q|}}\right)^{4r}B^{-4W}.
\]
If
\begin{equation}\label{eq:count-condition}
\sum_{p\in\mathcal S}E_p\log p
>W\log B+r\log(32r\sqrt{|q|}).
\end{equation}
then by Lemma~\ref{evaluation-residues} there exists distinct $f_1,f_2\in \mathcal L_D$ satisfying 
$$\|f_1\|_B,\|f_2\|_B\le H$$
and
$$f_1(\theta_n,\sigma\theta_n,\sigma\tau\theta_n)\equiv f_2(\theta_n,\sigma\theta_n,\sigma\tau\theta_n) \pmod{P\mathcal O}, \quad n\in \{0,\dots,N-1\}^{4}.$$
Taking $F=f_1-f_2$, we have
\[
\|\theta_n\|,\ \|\sigma\theta_n\|,\ \|\sigma\tau\theta_n\|\leq B,
\]
and hence
\[
\|F(\theta_n,\sigma\theta_n,\sigma\tau\theta_n)\|
\leq 2rH=\frac{P}{2}.
\]
On the other hand,
\[
F(\theta_n,\sigma\theta_n,\sigma\tau\theta_n)\in P\mathcal O.
\]
If this value were nonzero, its norm from $K$ to $\mathbb Q$ would be a nonzero integer divisible by $P^4$, while the absolute value of this norm is at most $(P/2)^4<P^4$. This proves~\eqref{eq:Fconds}.

It remains to verify that~\eqref{eq:PS} implies~\eqref{eq:count-condition}. Uniformly for $p\leq Q$, we have
\[
r=\frac{D^3}{384}+O(D^2),\qquad
W=\frac{5D^4}{6144}+O(D^3), \quad E_p=\frac{D^4}{1536p}+O(D^3).
\]
Consequently,
\[
\sum_{p\in\mathcal S}E_p\log p
=\frac{D^4}{1536}\sum_{p\in\mathcal S}\frac{\log p}{p}
+O(D^3Q)
\]
from which~\eqref{eq:count-condition} follows.
\end{proof}

\paragraph{Nonvanishing of polynomials}

We will combine Lemmas~\ref{analytic} and~\ref{pigeonholing} to produce a vanishing polynomial of small degree. 
Our contradiction will come from appealing to the following result of Masser~\cite{Masser}. This replaces the interpolation lemma and greedy selection used in OpenAI~\cite{OpenAI}.

Given a lattice $\Gamma\subseteq \mathbb{C}^n$ define
\[
 \mu(\Gamma)=\min_{L\subsetneq\mathbb C^n}
 \frac{\operatorname{rank}_{\mathbb Z}(\Gamma/(\Gamma\cap L))}
 {n-\dim_{\mathbb C}L}.
\]

\begin{lemma}\label{masser}
Let $n,M\ge1$ be integers, let $\gamma_1,\ldots,\gamma_s\in\mathbb C^n$, and let $\Gamma$ denote the lattice generated by $\gamma_1,\dots,\gamma_s$.  If a nonzero polynomial $F\in\C[x_1,\ldots,x_n]$ vanishes at every
$\sum_{j=1}^s m_j\gamma_j$, with integers $0\le m_j\le M$, then its total degree satisfies
\[
 \deg F\ge(M/n)^{\mu(\Gamma)}.
\]

\end{lemma}
Our next result gives the lower bound for $\mu(\Gamma)$ required in Lemma~\ref{masser}.
\begin{lemma}
\label{dim}
Let $\pi:\mathbb C^4\to\mathbb C^3$ be surjective. Suppose
$\ker\pi$ is not contained in any proper complex subspace
spanned by rational vectors. Then
\[
\mu\bigl(\pi(\mathbb Z^4)\bigr)\geq\frac43.
\]
\end{lemma}

\begin{proof}
Put $\Gamma=\pi(\mathbb Z^4)$. The hypothesis makes $\pi$
injective on $\mathbb Z^4$, so $\Gamma$ has rank $4$.

For a proper subspace $L\subset\mathbb C^3$, set
\[
W=\operatorname{span}_{\mathbb C}
\bigl(\pi^{-1}(L)\cap\mathbb Z^4\bigr).
\]
Since $W$ is proper and spanned by rational vectors our assumption implies
$\ker\pi\cap W=\{0\}$, and
\[
\operatorname{rank}_{\mathbb Z}(\Gamma\cap L)
=\dim_{\mathbb C}W
=\dim_{\mathbb C}\pi(W)
\leq\dim_{\mathbb C}L.
\]
Thus
\[
\mu(\Gamma)\geq
\min_{0\leq h<3}\frac{4-h}{3-h}=\frac43.
\]

\end{proof}

\begin{proof}[Proof of Theorem~\ref{thm:nosiegel}]
Suppose that 
$$(1-\beta)\log|q|<\varepsilon.$$ Fix a sufficiently large positive integer $\gamma$ and choose $N,B,D,Q$ as in Lemma~\ref{pigeonholing}. Mertens' estimate and Lemma~\ref{analytic} give
\[
\sum_{p\in\mathcal S}\frac{\log p}{p}
\ge\log Q-O\!\left(\log|q|+(1-\beta)(\log Q)^2\right).
\]
Since $\log Q=\frac43\log N+O(\log\log N)$ and $\log B=\log N+\frac12\log|q|+O(1)$, we have 
\[
\sum_{p\in\mathcal S}\frac{\log p}{p}
>\frac{13}{10}\log{B}
\]
provided $\gamma$ is sufficiently large, $\varepsilon$ sufficiently small in terms of $\gamma$, and $|q|$ sufficiently large. Lemma~\ref{pigeonholing} supplies a nonzero $F\in\mathcal L_D$ satisfying 
\begin{align}
\label{eq:FF}
F(\theta_n,\sigma\theta_n,\sigma\tau\theta_n)=0
\qquad\text{for every }n\in\{0,\ldots,N-1\}^4.
\end{align}
We apply Lemma~\ref{masser} with 
\[
\begin{aligned}
\gamma_1&=(1,1,1),\\
\gamma_2&=(\sqrt q,-\sqrt q,-\sqrt q),\\
\gamma_3&=(\sqrt2,\sqrt2,-\sqrt2),\\
\gamma_4&=(\sqrt{2q},-\sqrt{2q},\sqrt{2q}).
\end{aligned}
\]
Since
\[
(\theta_n,\sigma\theta_n,\sigma\tau\theta_n)
=n_1\gamma_1+\cdots+n_4\gamma_4,
\qquad 0\le n_j\le N-1
\]
by~\eqref{eq:FF}, the conditions of Lemma~\ref{masser} are satisfied.

Define the linear map \(\pi:\mathbb C^4\to\mathbb C^3\) by
\[
\pi(x_1,x_2,x_3,x_4)
=x_1\gamma_1+x_2\gamma_2+x_3\gamma_3+x_4\gamma_4.
\]
Since
\[
\det(\gamma_1,\gamma_2,\gamma_3)=4\sqrt{2q}\neq0,
\]
the map \(\pi\) is surjective and its kernel is spanned by the vector
\[
v=(\sqrt{2q},\sqrt2,-\sqrt q,-1).
\]
The coordinates of $v$ are linearly independent over $\Q$, thus $\ker \pi$ lies in no proper complex subspace spanned by rational vectors. Hence by Lemma~\ref{dim},
$$\mu(\Gamma)\geq4/3,$$ where $\Gamma$ is the lattice generated by  $\gamma_1,\dots,\gamma_4$.
Lemma~\ref{masser}, with $M=N-1$, now gives
$$\deg F\ge((N-1)/3)^{4/3}>D$$ contradicting that $F\in \mathcal L_D$.

\end{proof}

\bigskip
\noindent School of Science, UNSW Canberra, Canberra, Australia

\noindent \textit{Email:} \href{mailto:bryce.kerr@unsw.edu.au}{bryce.kerr@unsw.edu.au}
\end{document}